\documentclass{amsart}
\usepackage{graphicx,amsmath,amsfonts,amssymb,amsthm,mathtools,fullpage,url} 
\usepackage[dvipsnames]{xcolor}

\counterwithin{table}{section}
\counterwithin{figure}{section}

\theoremstyle{plain}
\newtheorem{theorem}{Theorem}[section]

\newtheorem{proposition}[theorem]{Proposition}
\theoremstyle{definition}
\newtheorem{definition}[theorem]{Definition}

\newtheorem{remark}[theorem]{Remark}

\newtheorem{conjecture}[theorem]{Conjecture}

\DeclareMathOperator{\Trace}{trd}
\DeclareMathOperator{\Norm}{nrd}
\DeclareMathOperator{\Gal}{Gal}

\DeclareMathOperator{\Discr}{Discr}
\DeclareMathOperator{\rank}{rank}

\def\Spec{\operatorname{Spec}}
\def\End{\operatorname{End}}

\def\Im{\operatorname{Im}}

\renewcommand{\phi}{\varphi}

\title{Minimal degree of an isogeny between a supersingular elliptic curve and its conjugate}

\author[Aubry]{Yves Aubry}
\address[Aubry]{Institut de Math\'ematiques de Toulon - IMATH, Universit\'e de Toulon, France}
\email{yves.aubry@univ-tln.fr}

\address[Aubry]{Institut de Math\'ematiques de Marseille - I2M, Aix Marseille Univ, UMR 7373 CNRS, France}
\email{yves.aubry@univ-amu.fr}

\address[Aubry]{Laboratoire GAATI, University of French Polynesia}
\email{yves.aubry@gaati.org}

\author[Oyono]{Roger Oyono}
\address[Oyono]{Laboratoire GAATI, University of French Polynesia}
\email{roger.oyono@upf.pf}

\author[Vincent]{Christelle Vincent}
\address[Vincent]{University of Vermont, Burlington, United States of America}
\email{christelle.vincent@uvm.edu}

\begin{document}

\begin{abstract}
Let $E$ be a supersingular elliptic curve defined over $\overline{\mathbb{F}}_p$ and $E^{(p)}$ be its conjugate. We give a bound on the minimal degree of an isogeny from $E$ to $E^{(p)}$ depending on $p$, and show that this bound is both asymptotically optimal as well as sharp in many cases. This bound is obtained by developing a new technique to compute the degree of certain isogenies from a supersingular elliptic curve to its conjugate, and we present extensive computations of the successive minima of the lattice containing these isogenies. Following this, we give several conjectures supported by the data we have obtained, including some on the set of primes $p$ for which the bound we give in this article is attained.
\end{abstract}



\thanks{The first and last author are grateful for the hospitality of the GAATI laboratory at the Universit\'{e} de Polyn\'{e}sie fran\c{c}aise, where they completed this work.
Aubry is partially supported by the French Agence Nationale de la Recherche through the Barracuda project under Contract ANR-21-CE39-0009-BARRACUDA.
Oyono is partially supported by a grant of STIC-AmSud entitled Crypto4All, number 24-STIC-04.
Vincent is supported by a Simons Foundation Travel Support for Mathematicians grant.}

\date{\today}

\maketitle

\section{Introduction}
Let $p$ be a prime, $\mathbb{F}_p$ be the finite field with $p$ elements, and $\overline{\mathbb{F}}_p$ be an algebraic closure of $\mathbb{F}_p$. 
Throughout, we write $E$ for a supersingular elliptic curve defined over $\overline{\mathbb{F}}_p$, and $E^{(p)}$ for its $\Gal(\mathbb{F}_{p^2}/\mathbb{F}_p)$-conjugate. In other words, if $\varphi(z)=z^p$ is the Frobenius homomorphism of $\mathbb{F}_p$,
$$E^{(p)}:=E\otimes_{\Spec(\mathbb{F}_p), \varphi}\Spec(\mathbb{F}_p).$$

In this article, we study the degrees of isogenies $E \to E^{(p)}$. In the simplest case, if $j(E) \in \mathbb{F}_p$, there is an isogeny $E \to E^{(p)}$ of degree $1$: $E$ and $E^{(p)}$ are isomorphic. In general, if $j(E) \in \overline{\mathbb{F}}_p \setminus \mathbb{F}_p$, it is natural to ask for the smallest integer $d$ such that there is an isogeny $E \to E^{(p)}$ of degree $d$; we denote this least value of $d$ by $\delta_E$ here.  
Specifically, we prove: 

\begin{theorem}\label{thm:main}
Let $p$ be a prime number and $E$ be a supersingular elliptic curve defined over 
$\overline{\mathbb F}_p$.
Then there exists an  isogeny from $E$ to $E^{(p)}$ of degree less than or equal to $\sqrt[3]{\frac{p}{2}}$.
\end{theorem}  

This theorem will follow immediately from Theorem \ref{better_bound}, and we will show that this bound is sharp in many cases. It is worth noticing that, by the discussion at the beginning of Section \ref{sec:Gross}, Theorem \ref{thm:main} is equivalent to the statement that if $E$ is a supersingular elliptic curve defined over $\overline{\mathbb{F}}_p$, $E$ has an inseparable endomorphism of degree less than or equal to $\sqrt[3]{\frac{p^4}{2}}$.

Our study of the degree of isogenies $E \to E^{(p)}$ leads us naturally to define the quantity
\begin{equation*}
\delta(p) = \max_{E \text{ ss}/ \overline{\mathbb{F}}_p} \left(\min_{\phi \colon E \to E^{(p)}} \deg(\phi) \right) = \max_{E \text{ ss}/ \overline{\mathbb{F}}_p}  \delta_E ,
\end{equation*}
where here $E$ ranges over the set of isomorphism classes of supersingular elliptic curves defined over $\overline{\mathbb{F}}_p$ and for each such curve $E$, $\phi$ ranges over the set of isogenies from $E$ to $E^{(p)}$. Our method of proof allows us to compute $\delta(p)$ for any prime $p$, and this article presents the extensive data we have obtained as well as the methods we have used to obtain them. These data have led us to formulate several conjectures, which conclude our article.

More precisely, the article is organized as follows: We begin in Section \ref{sec:preliminaries} by presenting the background necessary to our article as well as work done on related topics.  Before we can present the proof of Theorem \ref{thm:main}, we establish a connection between the isogenies $E \to E^{(p)}$ of small degree and the inseparable endomorphisms of $E$ of trace $0$ in Section \ref{sec:Gross}. Then in Section \ref{sec:minimal}, we use the relation between inseparable endomorphisms of trace $0$ of supersingular elliptic curves and the rank-$2$ sublattices of the Gross lattice of such curves established in \cite{AubryVincent} as well as a general theorem on elements represented by definite quadratic forms to prove Theorem \ref{thm:main}. Additionally we show that the exponent of $1/3$ in the bound that we obtain is asymptotically sharp, and that the bound of $\lfloor \sqrt[3]{\frac{p}{2}}\rfloor$ is sharp in several cases; this will give rise to the notion of \emph{wisde primes} presented in Definition \ref{def:wisde}. Then Section \ref{sec:methods} presents details on the computation of the quantity $\delta(p)$ as well as some tables containing salient data, and finally Section \ref{sec:data} delves deeper into the data we have obtained as well as presents the conjectures they suggest.

\section{Definitions and preliminaries}\label{sec:preliminaries}

\subsection{Related previous works}
Our work is related to a question posed in \cite{DLX}, where the authors study the least (nontrivial) degree of an isogeny $E \to E'$ for $E$ and $E'$ two elliptic curves defined over $\overline{\mathbb{F}}_p$.  
Here by ``nontrivial," we mean that isogenies of degree $1$ are excluded from consideration. 
 They show that if $E$ is an elliptic curve defined over a finite field (ordinary or supersingular), $E$ has a nontrivial endomorphism of degree less than or equal to $4$. This contrasts with the case of \emph{inseparable} endomorphisms of a supersingular elliptic curve considered here, as Remark \ref{rem:Yang} implies that the best upper bound on the least degree of an inseparable endomorphism grows like $p^{4/3}$. 
 
Turning our attention to isogenies from $E$ to its conjugate $E^{(p)}$, specifically in the supersingular case, one bound on the minimal degree can be obtained using the diameter of the $\ell$-isogeny graph of supersingular elliptic curves, where $\ell$ is a prime different from $p$. Since this diameter is on the order of $\log(p)$ by \cite[Theorem 79]{kohel}, this gives an upper bound on the least degree of an isogeny from $E$ to $E^{(p)}$ that is on the order of $p$. (Recall that if the distance between two $j$-invariants $j_1$ and $j_2$ in the $\ell$-isogeny graph of supersingular elliptic curves is $d$, then there exists an isogeny of degree $\ell^d$  between the corresponding curves.) Of course this bound is not specific to pairs of curves that are Galois conjugates, and is restricted to isogenies that are of degree given by a power of a fixed prime. It is therefore to be expected that it would be much larger than the bound obtained here. To support this expectation, experiments presented in \cite[Section 4.1]{ACLLNSS} suggest that in the $2$-isogeny graph, conjugate pairs are closer to each other than arbitrary pairs of supersingular elliptic curves. Relatedly, \cite[Theorem 3.9]{EHLMP} gives a lower bound on the number of supersingular elliptic curves defined over $\overline{\mathbb{F}}_p$ that are adjacent to their Galois conjugate in the supersingular $\ell$-isogeny graph when $\ell < \frac{p}{4}$.
 
In the setting of number fields, the main theorem of \cite{MW} asserts that if two elliptic curves $E$ and $E'$ defined over a number field are isogenous, then there exists an isogeny between them whose degree is bounded by $c \cdot w(E)^4$, where $w(E)$ denotes a height of the curve $E$ and $c$ is a constant depending only on the base field.

\subsection{Background}

We now turn our attention to giving some basic facts on supersingular elliptic curves. If $p$ is a prime and $E$ is a supersingular elliptic curve defined over $\overline{\mathbb{F}}_p$, we have that its $j$-invariant $j(E)$ belongs to $\mathbb{F}_{p^2}$. In particular, there are only finitely many $\overline{\mathbb{F}}_p$-isomorphism classes of supersingular elliptic curves for a given prime $p$. As stated in the introduction, a quantity of interest in this article is the least degree $\delta_E$ of an isogeny between a supersingular elliptic curve $E$ and its conjugate $E^{(p)}$, and we take the opportunity to give a formal definition here:

\begin{definition}\label{def:deltaE}
Let $p$ be a prime and $E$ be a supersingular elliptic curve defined over $\overline{\mathbb{F}}_p$. Then we define the value
\begin{equation*}
\delta_E = \min_{\phi \colon E \to E^{(p)} } \deg(\phi),
\end{equation*}
where the minimum is taken over all isogenies $\phi \colon E \to E^{(p)}$ from $E$ to its $\Gal(\mathbb{F}_{p^2}/\mathbb{F}_p)$-conjugate $E^{(p)}$.
\end{definition}

For all primes $p$, there exist supersingular elliptic curves with $j$-invariant in $\mathbb{F}_p$; again as stated in the introduction such curves are isomorphic to their conjugate and for those curves $\delta_E = 1$.

Let $B_{p,\infty}$ be the quaternion algebra over $\mathbb{Q}$ ramified at $p$ and $\infty$.
Then if $E$ is a supersingular elliptic curve defined over $\overline{\mathbb{F}}_p$, the geometric endomorphism ring of $E$, which we denote $\End(E)$, is isomorphic to a maximal order $\mathcal{O}$ of $B_{p,\infty}$. 
Letting $\overline{\cdot}$ be the conjugation involution on $B_{p,\infty}$, we define the reduced trace $\Trace(x) = x+\overline{x}$ and the reduced norm $\Norm(x) = x \overline{x}$ for $x \in B_{p,\infty}$ as well as an inner product $\frac{1}{2}\Trace(x\overline{y})$ for $x,y \in B_{p,\infty}$. 
To simplify the notation, for $\alpha \in \End(E)$ we also write $\Trace(\alpha)$ for its trace and $\Norm(\alpha)$ for its degree, as they are equal to the reduced trace, respectively the reduced norm, of the image of $\alpha$ under any isomorphism $\End(E) \otimes \mathbb{Q} \cong B_{p,\infty}$. We note that in this article, our computations take place in $B_{p,\infty}$; given a supersingular elliptic curve, obtaining an explicit endomorphism from an element of $B_{p,\infty}$ requires more work which is not necessary here as we are concerned with questions of existence. We give one explicit example in Section \ref{sec:example}.

Of importance in this article will be the \emph{Gross lattice} \cite{Gross} of a supersingular elliptic curve; we refer the reader to \cite{Goren-Love-short} and \cite{HKTV} for further reading on this topic. Let $\tau \colon B_{p,\infty} \to B_{p,\infty}$ be given by $\tau(x) = 2x - \Trace(x)$. Then for $\mathcal{O}$ an order in $B_{p,\infty}$, we write $\mathcal{O}^T$ for the image of $\mathcal{O}$ under $\tau$ and call $\mathcal{O}^T$ the \textbf{Gross lattice of $\mathcal{O}$}. If $E$ is a supersingular elliptic curve defined over $\overline{\mathbb{F}}_p$ with endomorphism ring isomorphic to a maximal order $\mathcal{O}$ of $B_{p,\infty}$, we say that $\mathcal{O}^T$ is the \textbf{Gross lattice of $E$}.

If $\mathcal{O}$ is an order in $B_{p,\infty}$, $\mathcal{O}^T$ is a lattice of rank $3$ contained in the lattice of elements of $\mathcal{O}$ of trace $0$. We define the \textbf{$i$th successive minimum} of a lattice $L$ in $B_{p,\infty}$ to be the least quantity $D_i$ such that
\begin{equation*}
 \rank \{ x \in L : \Norm(x) \leq D_i \} \geq i.
\end{equation*}
Since $\mathcal{O}^T$ is of rank $3$, it has a $\mathbb{Z}$-basis $\{\beta_1,\beta_2,\beta_3\}$ such that $\Norm(\beta_i) = D_i$; we call such a basis of $\mathcal{O}^T$ a \textbf{successive minimal basis}. If $L$ is any lattice of rank $3$ and $D_1 \leq D_2 \leq D_3$ are its successive minima, we have
\begin{equation*}
\det(L) \leq D_1D_2D_3 \leq 2 \det(L),
\end{equation*}
we call the second inequality the \textbf{Hermite bound} \cite[Theorem 2.6.8]{martinet13}.

The significance of the Gross lattice to this work is explained by the following result, which is \cite[Theorem 1.1]{AubryVincent}:
\begin{theorem}\label{maintheorem}\label{thm:AV}
Let $p$ be a prime and $E$ be a supersingular elliptic curve defined over $\overline{\mathbb{F}}_p$. Then for any positive integer $\ell$, there exists an endomorphism of $E$ of degree $\ell p$ and trace zero if and only if $\mathcal{O}^T$ contains a rank-2 sublattice of determinant $4\ell p$.
\end{theorem}

\section{Connection to sublattices of rank 2 of the Gross lattice}\label{sec:Gross}

It is well known that if $E$ is a supersingular elliptic curve defined over $\overline{\mathbb{F}}_p$, then there exists an endomorphism of $E$ of degree $np$ for $n$ a positive integer if and only if there is an isogeny $E \to E^{(p)}$ of degree $n$ (see for example \cite[Section 2.1]{AubryVincent} for a short justification of this fact). In this section, we argue that to study the isogeny $E \to E^{(p)}$ of least degree, it is sufficient to consider inseparable endomorphisms of $E$ of trace $0$. Then, by the relation between inseparable endomorphisms of $E$ of trace $0$ and rank-2 sublattices of the Gross lattice of $E$ of \cite{AubryVincent} which is recalled here in Theorem \ref{thm:AV}, we find that to study isogenies $E \to E^{(p)}$ of small degree, we may consider the determinants of such lattices. 

Let $\alpha \in \End(E)$ and $\tau(x) = 2x - \Trace(x)$. We have that 
\begin{equation*}
\Norm{(\tau(\alpha))} = 4 \Norm(\alpha) - \Trace(\alpha)^2,
\end{equation*}
and since $\Norm(\tau(\alpha)) \geq 0$, it follows that 
\begin{equation}\label{eq:inequality}
\Norm(\alpha) \geq \frac{1}{4}\Trace(\alpha)^2.
\end{equation} 

Now let $\alpha \in \End(E)$ be an inseparable endomorphism; since $E$ is supersingular this is equivalent to $p | \deg(\alpha)$. In addition, since $p$ is either inert or ramified in any imaginary quadratic field that embeds into $B_{p,\infty}$, an element of $B_{p,\infty}$ of norm divisible by $p$ must also have trace divisible by $p$. Combining this with the inequality \eqref{eq:inequality}, we thus have that if $\alpha$ is inseparable with $\Norm(\alpha) < \frac{p^2}{4}$, then $\Trace(\alpha) = 0$.

\begin{proposition}\label{prop:rank2sublattice}
If $E$ is a supersingular elliptic curve defined over $\overline{\mathbb F}_p$ then:
$$\delta_E=\min_{\Lambda \subset \mathcal{O}^T}\frac{\det(\Lambda)}{4p}$$
where the minimum runs over every rank-2 sublattice of the Gross lattice $\mathcal{O}^T$, where $\mathcal{O}$ is a maximal order of $B_{p,\infty}$ isomorphic to $\End(E)$.
\end{proposition}

\begin{proof}

Let ${\mathcal O}^{T}$ be the Gross lattice of $\mathcal O$, $D_1\leq D_2 \leq D_3$ be its successive minima and $\{\beta_1,\beta_2,\beta_3\}$ be a successive minimal basis of ${\mathcal O}^{T}$.
In addition, let $t_{12}:=\frac{1}{2} \Trace(\beta_1 \overline{\beta_2})$.

By Lemma 3.1 of \cite{CG14} and using the Hermite bound for a lattice of rank $3$, we have:
\begin{equation}\label{D1D2D3}
4p^2\leq D_1D_2D_3 \leq 8p^2
\end{equation}
hence 
$$D_1D_2-t_{12}^2\leq D_1D_2 \leq \frac{8p^2}{D_3}.$$

Moreover, by Proposition 3.12 of \cite{Goren-Love-short}, we know that $D_1D_2-t_{12}^2=4pn$ with $n\in {\mathbb N}\setminus \{0\}$.
So, $n\leq \frac{2p}{D_3}$ and thus $D_3\leq \frac{2p}{n}$.
Since $D_1 \leq D_2 \leq D_3$, 
$$ 4p^2 \leq D_1D_2D_3 \leq \left( \frac{2p}{n}\right)^3$$
which gives
$$n\leq \sqrt[3]{2p}.$$
Thus, we have shown that there exists a rank-2 sublattice of ${\mathcal O}^{T}$ of determinant $4pn$ with $n\leq \sqrt[3]{2p}$.

By Theorem 1.1 of \cite{AubryVincent}, there exists an inseparable endomorphism of $E$ of degree $pn$ (and trace $0$) with $n\leq \sqrt[3]{2p}$, and therefore an isogeny  $E \to E^{(p)}$ of degree  $n\leq \sqrt[3]{2p}$. From this it follows that $\delta_E \leq \sqrt[3]{2p}$.

Now let $\phi \colon E \to E^{(p)}$ be an isogeny from $E$ to $E^{(p)}$ of least degree. We have thus that $\deg(\phi) \leq \sqrt[3]{2p}$. This isogeny corresponds to an endomorphism $\psi$ of $E$ of degree less than or equal to $\sqrt[3]{2p^4}$. We have that $\sqrt[3]{2p^4} < \frac{p^2}{4}$ when $p > 11$, and therefore in this case we must have $\Trace(\psi) = 0$, as argued immediately before the statement of this proposition. Therefore, the isogeny $E \to E^{(p)}$ of least degree must correspond to an inseparable endomorphism of $E$ of trace $0$, which in turn corresponds to a sublattice of $\mathcal{O}^T$ of rank $2$.

Finally, we handle the case of $p \leq 11$. In each case, every supersingular elliptic curve defined over $\overline{\mathbb{F}}_p$ has $j$-invariant in $\mathbb{F}_p$ and therefore $\delta_E =1$. By \cite[Proposition 3.2.2]{HKTV}, in each of these cases, the lattice spanned by $\beta_1$ and $\beta_2$ in $\mathcal{O}^T$, where $\beta_1$ and $\beta_2$ attain the first two successive minima of $\mathcal{O}^T$, has determinant $4p$ and therefore the result holds.
\end{proof}

\begin{remark}
Note that this proof does give a first bound $\delta_E \leq \sqrt[3]{2p}$ for $E$ a supersingular elliptic curve defined over $\overline{\mathbb{F}}_p$, which we will improve in the next section.
\end{remark}

\section{On the minimal degree of an isogeny $E \to E^{(p)}$}\label{sec:minimal}

In this section, we study the quartic form giving the determinant of a sublattice of $\mathcal{O}^T$ of rank $2$. After a change of variable, we can obtain a ternary quadratic form which we denote $Q$ and which represents the same integers as the original determinant form. Then, using a result of Cassels, we can give a sharp bound on the integers represented by $Q$, which by Proposition \ref{prop:rank2sublattice} gives us a bound on $\delta_E$. Finally we give several examples where the bound is sharp.

\subsection{Bounding the minimal degree}

Let $\gamma_1$ and $\gamma_2$ be two elements of $\mathcal{O}^T$ and denote by $\Lambda$ the lattice that they generate. 
As usual we write $\{\beta_1,\beta_2,\beta_3\}$ for a 
successive minimal basis for $\mathcal{O}^T$, then 
$$\gamma_1 = a_1 \beta_1 + a_2 \beta_2 + a_3 \beta_3$$
 and 
 $$\gamma_2 = b_1 \beta_1 + b_2 \beta_2 + b_3 \beta_3$$
  for some integers $a_i$ and $b_i$, for $i = 1, 2, 3$. 
  
As before we set
$D_i:=\Norm(\beta_i)$
and
$t_{ij}:=\frac{1}{2}\Trace(\beta_i\bar\beta_j)$, then the Gram matrix for the basis $\{\beta_1,\beta_2,\beta_3\}$ is
  
\begin{equation}\label{Gram}
G=
\left(
\begin{matrix}
D_1& t_{12} & t_{13}\\
t_{21} & D_2 & t_{23}\\
t_{31} & t_{32} & D_3\cr
\end{matrix}
\right).
\end{equation}
  
With these conventions, a calculation shows that
\begin{equation*}
\Norm(\gamma_1) = D_1 a_1^2 + 2 t_{12}a_1a_2 + 2 t_{13}a_1a_3+D_2a_2^2 
+2t_{23}a_2a_3+D_3a_3^2
\end{equation*}
and
\begin{equation*}
\Norm(\gamma_2) = D_1 b_1^2 + 2 t_{12}b_1b_2 + 2 t_{13}b_1b_3+D_2b_2^2 
+2t_{23}b_2b_3+D_3b_3^2
\end{equation*}
as well as
\begin{equation*}
\frac{1}{2}\Trace(\gamma_1 \overline{\gamma}_2) = 
D_1a_1b_1+D_2a_2b_2+D_3a_3b_3
+t_{12}a_1b_2+t_{12}a_2b_1+t_{13}a_1b_3+t_{13}a_3b_1+t_{23}a_2b_3+t_{23}a_3b_2.
\end{equation*}

The determinant of $\Lambda$ is then given by the expression $\Norm(\gamma_1)\Norm(\gamma_2) - \frac{1}{4} \Trace(\gamma_1\overline{\gamma}_2)^2$, which is a homogenous polynomial $\det \Lambda$ of degree 4 in 6 variables $a_1,a_2,a_3,b_1,b_2,b_3$.

After performing the change of variables
\begin{equation}\label{eq:variables}
\left\{
\begin{matrix}
x=a_1b_2-a_2b_1\\
y=a_1b_3-a_3b_1\\
z= a_2b_3-a_3b_2\cr
\end{matrix}
\right.
\end{equation}
we obtain that $\det \Lambda$ is equal to the quadratic form $Q(x,y,z)\in{\mathbb Z}[T]$ given by:
\begin{align}\label{eq:Q}
Q = &(D_1D_2-t_{12}^2)x^2 + (D_1D_3-t_{13}^2)y^2 + (D_2D_3-t_{23}^2)z^2 \\
& +2(D_1t_{23}-t_{12}t_{13})xy - 2(D_2t_{13}-t_{12}t_{23})xz + 2(D_3t_{12}-t_{13}t_{23})yz. \notag
\end{align}

We have that $\det \Lambda$ and $Q$ represent the same integers, by which we mean that there are $a_i,b_i \in \mathbb{Z}$ such that $\det \Lambda (a_1,a_2,a_3,b_1,b_2,b_3) = n$ if and only if there are integers $x,y,z$ such that $Q(x,y,z) = n$. 
Indeed, given values $a_i, b_i \in \mathbb{Z}$ for $i= 1,2,3$, we can obtain values $x,y,z \in \mathbb{Z}$ straightforwardly from the change of variables given. Conversely, as argued in the proof of Theorem 3.3 of \cite{AubryVincent}, given integers $x,y,z$, we can always find integers $a_i, b_i$ for $i= 1,2,3$ satisfying the change of variables above. From this it follows that $\det \Lambda$ and $Q$ represent the same values.

Therefore we may now turn our attention to giving an upper bound on the least positive integer $n$ such that $Q$ represents $n$. For this we may use a result of Cassels bounding this value in terms of the discriminant of the ternary quadratic form.

As given on p.\ 31 of \cite{Cassels}, we have:
\begin{equation}\label{Discriminant}
\Discr(Q)=
\left\vert
\begin{matrix}
D_1D_2-t_{12}^2 & D_1t_{23}-t_{12}t_{13} & -(D_2t_{13}-t_{12}t_{23})\\
D_1t_{23}-t_{12}t_{23} & D_1D_3-t_{13}^2 & D_3t_{12}-t_{13}t_{23}\\
-(D_2t_{13}-t_{12}t_{23}) & D_3t_{12}-t_{13}t_{23} & D_2D_3-t_{23}^2\cr
\end{matrix}
\right\vert.
\end{equation}

\begin{proposition}\label{Discriminant}
For the matrix $G$ as in equation \eqref{Gram} and the ternary quadratic form given in \eqref{eq:Q}, we have:
$$\Discr(Q)=\det(G)^2=16p^4.$$
\end{proposition}

\begin{proof}
A tedious calculation shows that $\Discr(Q)=\det(G)^2$, and by \cite[Lemma 3.1]{CG14} we have $\det(G) = 4p^2$.
\end{proof}

We can now prove the main theorem of this article:

\begin{theorem}\label{better_bound}
Let $p$ be a prime number and $E$ be a supersingular elliptic curve defined over 
$\overline{\mathbb F}_p$, then
$$\delta_E\leq \sqrt[3]{\frac{p}{2}}.$$
\end{theorem}  

\begin{proof}
Cassels proved in Theorem III page 33 of \cite{Cassels} that if $f(x)=\sum f_{ij}x_ix_j$ with $f_{ij}=f_{ji}$
is a positive definite ternary quadratic form, then there is an integral vector $u\not=0$ such that
$$f(u)\leq \sqrt[3]{2\Discr(f)}$$
where $\Discr(f)=\det(f_{ij})$ is the discriminant of $f$.

Now let $\mathcal{O}$ be a maximal order of $B_{p,\infty}$ isomorphic to $\End(E)$. Cassels's result implies, together with Proposition \ref{Discriminant}, that the quadratic form $Q$ represents an integer $n$ such that $n\leq \sqrt[3]{32p^4} = 4p \sqrt[3]{\frac{p}{2}}$.
Thus $\mathcal{O}^T$ contains a rank-2 sublattice of determinant $\leq 4p \sqrt[3]{\frac{p}{2}}$, and therefore by Proposition \ref{prop:rank2sublattice}, we conclude that $\delta_E\leq \sqrt[3]{\frac{p}{2}}$. 
\end{proof}

 \begin{remark}\label{rem:Yang}
The bound of Theorem \ref{better_bound} is best possible asymptotically: Indeed, Yang  proved in \cite[Proposition 1.4]{Yang} that, if there are positive constants $\theta$ and $C$ such that every supersingular elliptic curve over $\overline{\mathbb F}_p$ can be lifted to a CM elliptic curve over some number field with CM by the quadratic ring of discriminant $D$ with $D\leq C p^{\theta}$, then $\theta\geq \frac{2}{3}$.
But a supersingular elliptic curve $E$ can be lifted to an elliptic curve defined over a number field with CM by the quadratic ring of discriminant $D$
if and only if its Gross lattice represents $D$ by \cite[Proposition 3.7]{Goren-Love-short} and the remark immediately following conditions (i) and (ii) of \cite[2.1.5]{CCO}.

Hence, if $\theta<\frac{2}{3}$, then for any constant $C$, there exist a prime $p$ and  a supersingular elliptic curve $E$ defined over $\overline{\mathbb F}_p$ with $D_1>Cp^{\theta}$.
Letting $D_i$ be the $i$th successive minimum of the Gross lattice of $E$ and $t_{12} = \frac{1}{2} \Trace(\beta_1 \overline{\beta}_2)$ as before, by \cite[Lemma 2.6.1]{HKTV} we have that $|t_{12}| \leq \frac{D_1}{2}$, and therefore
\begin{equation*}
D_1D_2 - t_{12}^2 \geq \frac{3D_1^2}{4} > \frac{3C^2}{4}p^{2\theta}.
\end{equation*}
Since $D_1D_2 - t_{12}^2$ is the determinant of a sublattice of the Gross lattice of rank $2$, by Proposition \ref{prop:rank2sublattice} for this elliptic curve $E$ there is an isogeny $E \to E^{(p)}$ of degree strictly greater than
\begin{equation*}
\frac{3C^2}{16}p^{2\theta-1}.
\end{equation*}

Letting $\theta = \frac{\eta +1}{2}$ and $C = \sqrt{\frac{4C'}{3}}$, we deduce thus from Yang's result that if $\eta <\frac{1}{3}$, then for any constant $C'>0$ there exists a prime number $p$ and a supersingular elliptic curve $E$ defined over $\overline{\mathbb F}_p$ with an isogeny $E\longrightarrow E^{(p)}$ of degree at least $C'p^{\eta}$.
 \end{remark}

\subsection{Sharpness of the bound} 

Having given this bound and established that it is asymptotically sharp, it is natural to ask if it can be attained. In this section we answer this question in the affirmative with several examples, beginning with one detailed case.

\subsubsection{An example where $p = 101,051$}\label{sec:example}

In this section, let $p = 101,051$, and consider the quaternion algebra $B_{p,\infty}$ ramified only at $p$ and $\infty$. This quaternion algebra can be written using the basis $\{1,i,j,k\}$ where $i^2 = -1$, $j^2 = -p$ and $ij = k$ as usual, and for this prime, $\lfloor \sqrt[3]{\frac{p}{2}}\rfloor = 36$. 

Let $\mathcal{O}$ be the maximal order of $B_{p,\infty}$ with reduced basis
\begin{equation*}
\left[1, \frac{1}{2} + \frac{59093}{2662}i - \frac{175}{2662}j + \frac{65}{2662}k,  \frac{2089}{242}i + \frac{5}{121}j - \frac{21}{242}k, \frac{37902}{1331}i + \frac{62}{1331}j + \frac{15}{1331}k\right].
\end{equation*}

Since $p \equiv 3 \pmod{4}$, we have $\mathbb{F}_{p^2} \cong \mathbb{F}_p[x]/(x^2+1)$; write $\alpha$ for an element of $\mathbb{F}_{p^2}$ such that $\alpha^2 = -1$. Now let $E$ be the supersingular elliptic curve defined over $\mathbb{F}_{p^2}$ with $j$-invariant $j = 67611\alpha + 42318$ and equation given by
\begin{equation*}
y^2 = x^3 + (29936\alpha+99804)x + (16688\alpha+54970).
\end{equation*}

Then we have that both $E$ and its Galois conjugate satisfy $\End(E) \cong \mathcal{O}$ and $\delta_E = 36$. Here we have used the software provided as a companion to \cite{deuringforthepeople} to compute the equation of $E$ from a basis of its endomorphism ring and \cite{sagemath} to obtain the equation of the curve.

\subsubsection{All examples for $p \leq 22,000$}

In addition to this example, we have performed an exhaustive search for all primes $p \leq 22,000$, computing for each prime a basis for each maximal order $\mathcal{O}$ of $B_{p,\infty}$ and the value $\delta_E$ of a supersingular elliptic curve $E$ such that $\End(E) \cong \mathcal{O}$. In Table \ref{table:smallwisdeprimes}, we list the values of $p$ such that there is $E$ defined over $\overline{\mathbb{F}}_p$ with $\delta_E = \lfloor \sqrt[3]{\frac{p}{2}} \rfloor$. In other words, for these primes there is a supersingular elliptic curve with no isogeny to its Galois conjugate of degree smaller than the largest possible value allowed by Theorem \ref{thm:main}, and in those cases the bound we provide is sharp.

\begin{table}[h] 
\caption{Primes between $2$ and $22,000$ such that there exists a supersingular elliptic curve $E$ defined over $\overline{\mathbb{F}}_p$ with $\delta_E = \lfloor \sqrt[3]{\frac{p}{2}} \rfloor$}
\begin{tabular}{ c | l }
$\delta_E = \lfloor \sqrt[3]{\frac{p}{2}} \rfloor$ & $p$ \\
\hline
1 & 2, 3, 5, 7, 11, 13 \\
2 & 37, 43, 53 \\
3 & 101, 107, 109, 113, 127 \\
4 & 211, 227, 239, 241 \\
5 & 331, 379, 401, 409, 421, 431 \\
6 & 617, 641, 653, 673, 683 \\
7 & 857, 937, 953, 977, 983, 1013 \\
8 & 1279, 1399, 1447 \\
9 & 1811, 1871, 1949, 1951, 1979, 1997, 1999 \\
10 & -- \\
11 & 3109, 3181, 3229, 3253 \\
12 & -- \\
13 & 5279, 5387, 5417, 5419, 5483 \\
14 & 6361, 6481\\
15 & 8081, 8111\\
16 & 9623 \\
17 & 11287, 11383, 11467, 11503\\
18 & --\\
19 & 15101, 15541, 15901 \\
20 & 17851, 18379 \\
21 & 20219, 20747, 20879, 21143, 21187
\end{tabular}
\label{table:smallwisdeprimes}
\end{table}

\section{Methods and computations}\label{sec:methods}
 
Having established a bound on the value $\delta_E$, we now turn our attention to a technique used to find primes $p$ and supersingular elliptic curves defined over $\overline{\mathbb{F}}_p$ such that $\delta_E = \lfloor \sqrt[3]{\frac{p}{2}}\rfloor$. Following this, we describe how we computed $\delta_E$ for each such curve and the extent of the calculations performed in the preparation of this manuscript.
 
 \subsection{Sieving for examples}\label{sec:sieving}
 
As $p$ grows, the time necessary to compute $\delta_E$ for every supersingular elliptic curve defined over $\overline{\mathbb{F}}_p$ grows quickly (see Section \ref{sec:compmethods} for the time it takes to perform the computation of $\delta_E$ for all supersingular elliptic curves $E$ defined over $\overline{\mathbb{F}}_{263429}$, for example), and therefore it becomes prohibitive to conduct an exhaustive search for primes $p$ such that there is a supersingular elliptic curve $E$ defined over $\overline{\mathbb{F}}_p$ with $\delta_E = \lfloor \sqrt[3]{\frac{p}{2}}\rfloor$ . In this section we give a method to reduce the number of primes to consider, by giving a computational necessary condition on $p$ for such a supersingular elliptic curve to exist.
 
 Let $E$ be a supersingular elliptic curve defined over $\overline{\mathbb{F}}_p$ and $\mathcal{R}$ be the sublattice of isogenies $\varphi\colon E \to E^{(p)}$ such that if $\hat{\varphi} \colon E^{(p)} \to E$ is the dual isogeny of $\varphi$ and $F \colon E \to E^{(p)}$ is the degree-$p$ Frobenius isogeny, then $\hat{\varphi} \circ F \in \End(E)$ is of trace $0$. Since by \cite{AubryVincent}, the sublattice of $\End(E)$ of inseparable elements of trace $0$ has rank $3$, and $\varphi \mapsto \hat{\varphi} \circ F$ is a bijection between isogenies from $E$ to $E^{(p)}$ and inseparable endomorphisms of $E$, then $\mathcal{R}$ is also a lattice of rank $3$. Furthermore, again by \cite{AubryVincent}, for $Q$ as in equation \eqref{eq:Q}, the degree form on $\mathcal{R}$ is equal to $Q' = \frac{Q}{4p}$ and is of discriminant $\frac{p}{4}$. If $\{\varphi_1,\varphi_2,\varphi_3\}$ is a successive minimal basis for the quadratic module $(\mathcal{R}, Q')$, we denote the Gram matrix of such a basis by
 \begin{equation*}
N =  \begin{pmatrix}
 N_1 & \frac{x}{2} & \frac{y}{2} \\
  \frac{x}{2} & N_2 &  \frac{z}{2} \\
   \frac{y}{2} &  \frac{z}{2} &N_3
 \end{pmatrix},
 \end{equation*}
 where $x,y,z \in \mathbb{Z}$ and $N_1$ is the least degree of an isogeny $E \to E^{(p)}$, as argued in Section \ref{sec:Gross}. By an elementary computation, we obtain that
 \begin{equation}\label{eq:detN}
 p = 4 \det (N) = xyz -N_3x^2 - N_2y^2 - N_1z^2 + 4N_1N_2N_3 .
 \end{equation}
 
By \cite[Lemma 2.3.4]{HKTV} and adapting the proof of \cite[Lemma 2.6.1]{HKTV} to this case, since this is a lattice of rank $3$, we may ensure that $0 \leq x,y$ by negating the last two elements of our successive minimal basis if necessary and we have that $x,y \leq N_1$ and $|z| \leq N_2$. Furthermore, using the Hermite bound we have $\frac{p}{4}\leq N_1N_2N_3 \leq \frac{p}{2}$. If we suppose that $N_1 = \lfloor \sqrt[3]{\frac{p}{2}} \rfloor$, and using that $N_1 \leq N_2$, this forces 
 \begin{equation*}
\left\lfloor \sqrt[3]{\frac{p}{2}} \right \rfloor \leq N_2 \leq N_3 \leq \frac{p}{ \lfloor \sqrt[3]{\frac{p}{2}} \rfloor^2}.
 \end{equation*}
 
Putting this together, we have that $p$ is a prime such that there is a supersingular elliptic curve $E$ defined over $\overline{\mathbb{F}}_p$ with $\delta_p = \lfloor \sqrt[3]{\frac{p}{2}} \rfloor$ only if there is a tuple of integers $(N_1,N_2,N_3,x,y,z)$
satisfying equation \eqref{eq:detN} as well as the conditions
\begin{equation}\label{eq:constraints}
\begin{cases}
N_1 = \lfloor \sqrt[3]{\frac{p}{2}}  \rfloor,\\
\left \lfloor \sqrt[3]{\frac{p}{2}} \right \rfloor \leq N_2 \leq \frac{p}{ \lfloor \sqrt[3]{\frac{p}{2}} \rfloor^2},\\
\max\left(N_2, \left \lceil \frac{p}{4N_1N_2}\right \rceil\right) \leq N_3 \leq \left\lfloor \frac{p}{2N_1N_2}\right\rfloor, \\
0 \leq x,y \leq N_1,\\
|z| \leq N_2,
\end{cases}
\end{equation}
which require a (quick, see Section \ref{sec:compmethods} for timings) finite computation to verify. 

We stress that the existence of a tuple of integers satisfying equation \eqref{eq:detN} and the constraints \eqref{eq:constraints} for a given prime $p$ does not guarantee the existence of a supersingular elliptic curve $E$ defined over $\overline{\mathbb{F}}_p$ with $\delta_E = \lfloor \sqrt[3]{\frac{p}{2}} \rfloor$. For example, if $p = 22,273$, we obtain the Gram matrix
\begin{equation*}
N = \begin{pmatrix}
 22 & \frac{15}{2} & 8 \\
  \frac{15}{2} & 22 &  -9 \\
8 &  -9 &23
\end{pmatrix}
\end{equation*}
for a putative quadratic module $(\mathcal{R},Q')$, but there is no supersingular elliptic curve realizing this quadratic module.

By checking only primes for which we find solutions to \eqref{eq:detN} and \eqref{eq:constraints}, we can extend Table \ref{table:smallwisdeprimes} to contain the same data for primes $22,000 \leq p \leq 265,207$ in Table \ref{table:mediumwisdeprimes}.

\begin{table}[h] 
\caption{Primes between $22,000$ and $265,207$ such that there exists a supersingular elliptic curve $E$ defined over $\overline{\mathbb{F}}_p$ with $\delta_E = \lfloor \sqrt[3]{\frac{p}{2}} \rfloor$}
\begin{tabular}{ c | l }
$\delta_E = \lfloor \sqrt[3]{\frac{p}{2}} \rfloor$ & $p$ \\
\hline
22 & 24197 \\
23 & 27277, 27527,  27529\\
24 & 31079, 31151 \\
25 & 34919, 35023 \\
26 & 38557, 39181, 39313 \\
27 & 42197, 43037 \\
28 & -- \\
29 & 52051, 53731, 53849 \\
30 & 59369 \\
31 & -- \\
32 & 70423, 70687, 71647 \\
33 & 77351, 78437, 78439 \\
34 & 84421 \\
35 & 91621, 91909, 92761, 93131, 93133\\
36 & 101051
\end{tabular}
\quad
\begin{tabular}{ c | l }
$\delta_E = \lfloor \sqrt[3]{\frac{p}{2}} \rfloor$ & $p$ \\
\hline
37 & 108187 \\
38 & 116689, 118369\\
39 & 124601, 127241 \\
40 & -- \\
41 & 144439, 147799\\
42 & 157037 \\
43 & -- \\
44 & -- \\
45 & 193799, 194443 \\
46 & 205297 \\
47 & 218737 \\
48 & 232751, 234959\\
49 & 249749 \\
50 & 264949\\
\phantom{} & 
\end{tabular}

\label{table:mediumwisdeprimes}
\end{table}

\subsection{Computational methods}\label{sec:compmethods}

To calculate $\mathbb{Z}$-bases for maximal orders in $B_{p,\infty}$, we use the algorithm given by Kirschmer and Voight~\cite{KV10} and available in Magma~\cite{magma}. For the prime $p = 263,429$, which is one of the largest prime for which we performed this computation, the computation for a basis for each maximal order in $B_{263429, \infty}$ took 6885.410 seconds (1 hour and 54 minutes) on an Apple computer with M2 chip and 8 GB of RAM using Magma V2.29-2 on macOS Tahoe. To complete the computations of bases of quaternionic maximal orders presented in this article, we used the Ai'a server supported by the GAATI and GEPASUD laboratories at the Universit\'{e} de Polyn\'{e}sie fran\c{c}aise, all of whose technical specifications can be found at \url{gaati.org/aia}. This server ran Magma V2.28-23 at the time of our computations.

Once we obtained a basis for each maximal order $\mathcal{O}$, we computed a successive minimal basis for the Gross lattice $\mathcal{O}^T$, as well as its Gram matrix using code written by Korpal for \cite{HKTV} and available at \cite{HKTVrepo}. To compute this successive minimal basis, this code applies an implementation of Eisenstein reduction~\cite[Theorem 103]{Dickson} written by Rama and available within \texttt{SageMath}\footnote{\url{https://github.com/sagemath/sage/blob/develop/src/sage/quadratic_forms/ternary_qf.py}}. Eisenstein reduction yields a successive minimal basis for lattices of rank $3$ (the only case under consideration in this article) since an Eisenstein-reduced basis gives a reduced fundamental parallelepiped by \cite[pp.\ 162-163]{Dickson}, and therefore such a basis can shown to be Minkowski-reduced \cite{Minkowski}. Finally, in dimension smaller than $4$, a Minkowski-reduced basis attains the successive minima of a lattice by \cite{vanderWaerden}.

Computing the Gram matrix of the basis we obtain, allows us to determine values $D_1, D_2, D_3, t_{12}, t_{13}$ and $t_{23}$ which in turn give us the quadratic form $Q$ of equation \eqref{eq:Q}, from which we can compute the quadratic form $Q' = \frac{Q}{4p}$. Again we may use Eisenstein reduction to obtain a successive minimal basis of the module $(\mathcal{R},Q')$ of Section \ref{sec:sieving}, and the first successive minimum $N_1$ is the least degree $\delta_E$ of an isogeny from $E$ to $E^{(p)}$ for the isomorphism class(es) of elliptic curve(s) with endomorphism ring isomorphic to $\mathcal{O}$. This last step is done using our own functions written using \texttt{SageMath}~\cite{sagemath}, and our code is available for review on GitHub at~\cite{repo}. 

Again for the prime $p = 263,429$, once the bases of every maximal order of $B_{263429, \infty}$ had been computed, on the same Apple computer as above it took less than 35 seconds to compute the value of $\delta_E$ for all of the supersingular elliptic curves $E$ defined over $\overline{\mathbb{F}}_{263429}$.

Sieving for primes $p$ for which there may be a supersingular elliptic curve $E$ defined over $\overline{\mathbb{F}}_p$ with $\delta_E = \lfloor \sqrt[3]{\frac{p}{2}}  \rfloor$ is done by a brute-force search implemented in Magma and also available on GitHub at~\cite{repo}. Once more, on the Apple computer described above, for the prime $p = 265,231$ the brute-force search took 1.6 seconds, which is a worst-case scenario as there is no solution (and hence the calculation does not break before checking every possible integer satisfying the constraints of \eqref{eq:constraints}).
 
To show how useful the sieving step is, we have checked for the existence of integer solutions to equation \eqref{eq:detN} and constraints \eqref{eq:constraints} for all primes $p \leq 1,351,019$. There are $103,619$ such primes, and we have found $10,898$ primes satisfying the necessary condition. Thus this condition reduces the search space by approximatively $90\%$. We have computed the values of $\delta_E$ for every supersingular elliptic curve defined over $\overline{\mathbb{F}}_p$ for $2193$ of these primes (see Section \ref{sec:extent} for a more precise description of these primes). Of these, $45$ primes did in fact have a supersingular elliptic curve with $\delta_E = \lfloor \sqrt[3]{\frac{p}{2}}  \rfloor$, as listed in Table \ref{table:mediumwisdeprimes} above.

\subsection{Extent of the calculations}\label{sec:extent}

In the preparation of this manuscript, we have computed $\delta_E$ for every supersingular elliptic curve defined over $\overline{\mathbb{F}}_p$ for the following ranges of primes:
\begin{enumerate}
\item every prime $p \leq 22,000$;
\item every prime $93,312 \leq p \leq 101,306$, corresponding to every prime such that $\left \lfloor \sqrt[3]{\frac{p}{2}} \right \rfloor = 36$;
\item finally, for primes $22,000 \leq p \leq 93,312$ and $101,306 \leq p \leq 265,207$, for every prime for which there is an integer solution to equation \eqref{eq:detN} satisfying the conditions given in \eqref{eq:constraints}.
\end{enumerate}

The set of primes $p$ such that $p \leq 22,000$ and there is a supersingular elliptic curve $E$ defined over $\overline{\mathbb{F}}_p$ with $\delta_E= \left \lfloor \sqrt[3]{\frac{p}{2}} \right \rfloor$ is given in Table \ref{table:smallwisdeprimes} above. For the primes in the range $93,312 \leq p \leq 101,306$, there is a single prime such that there is a supersingular elliptic curve defined over $\overline{\mathbb{F}}_p$ with $\delta_E = \left \lfloor \sqrt[3]{\frac{p}{2}} \right \rfloor$. This is $p = 101,051$ and one of the two curves defined over $\overline{\mathbb{F}}_{101051^2}$ with $\delta_E = 36$ is given in Section \ref{sec:example}; since this curve has $j$-invariant outside of $\mathbb{F}_{101051}$, both this curve and its Galois conjugate satisfy $\delta_E = 36$ and correspond to a unique maximal order of $B_{101051,\infty}$ attaining the bound. Finally, Table \ref{table:mediumwisdeprimes} contains all primes $22,000 \leq p \leq 265,207$ such that there is a supersingular elliptic curve $E$ defined over $\overline{\mathbb{F}}_p$ with $\delta_E= \left \lfloor \sqrt[3]{\frac{p}{2}} \right \rfloor$, including $p = 101,051$.

In total, there are $23,247$  primes less than or equal to $265,207$, and $2610$ of these primes admit integer solutions to \eqref{eq:detN} satisfying the constraints \eqref{eq:constraints}. Overall, using both the sieve and exhaustive computations, we found that the number of primes $p$ less than or equal to $265,207$ such that there is a supersingular elliptic curve $E$ defined over $\overline{\mathbb{F}}_p$ with $\delta _E = \lfloor \sqrt[3]{\frac{p}{2}}  \rfloor$ is equal to 119. Hence approximately $0.5\%$ of primes less than or equal to $265,207$ admit a supersingular elliptic curve attaining the bound given in this article.

\section{Data analysis and conjectures}\label{sec:data}

Motivated by the work above, it is natural to wish to study the primes $p$ such that there is a supersingular elliptic curve $E$ defined over $\overline{\mathbb{F}}_p$ with $\delta_E = \lfloor \sqrt[3]{\frac{p}{2}} \rfloor$. To this end, in Section \ref{sec:defdata} we define the quantity $\delta(p)$ for $p$ a prime and present the data we have obtained on this quantity. This naturally leads us to the definition of elliptic curves \emph{without inseparable small-degree endomorphisms}, and call the primes $p$ for which there exist such curves defined over $\overline{\mathbb{F}}_p$ \emph{wisde primes}. Finally in Section \ref{sec:conjectures} we present conjectures on the functions $\delta_E$ and $\delta(p)$ and on the finiteness of the set of wisde primes 
that seem supported by the data we have obtained.

\subsection{Definitions and data} \label{sec:defdata}

To lighten our prose, let us introduce the following two definitions:
\begin{definition}\label{def:wisde}
Let $E$ be a supersingular elliptic curve defined over $\overline{\mathbb{F}}_p$. If $\delta_E = \lfloor \sqrt[3]{\frac{p}{2}} \rfloor$, we say that $E$ is \textbf{without inseparable small-degree endomorphisms}, or \emph{wisde} for short. If $p$ is a prime such that there exists a supersingular elliptic curve defined over $\overline{\mathbb{F}}_p$ that is wisde, then we say that $p$ is a \textbf{wisde prime}.
\end{definition}

We will also need:
\begin{definition}
Let $p$ be a prime, we define the value
\begin{equation*}
\delta(p) = \max_{E \text{ ss } /\overline{\mathbb{F}}_p} \delta_E,
\end{equation*}
where the maximum is taken over all isomorphism classes of supersingular elliptic curves defined over $\overline{\mathbb{F}}_p$.
\end{definition}

Therefore $p$ is a wisde prime if and only if $\delta(p) = \lfloor \sqrt[3]{\frac{p}{2}} \rfloor$. In addition, in Section \ref{sec:example} we showed that $p = 101,051$ is a wisde prime with $\lfloor \sqrt[3]{\frac{p}{2}} \rfloor =36$ (in fact, Figure \ref{figure3} shows that it is the only wisde prime with $93,312 \leq p \leq 101,306$), Table \ref{table:smallwisdeprimes} gives a complete list of wisde primes that are less than $22,000$, and Table \ref{table:mediumwisdeprimes} gives a complete list of wisde primes with $22,000 \leq p \leq 93,312$ and $101,306 \leq p \leq 265,207$. We present the data for primes $p \leq 22,000$ as graphs below. From Figures \ref{figure0}, \ref{figure1} and \ref{figure2} or from Table \ref{table:smallwisdeprimes}, we see that there is a wisde prime $p$ with $\lfloor \sqrt[3]{\frac{p}{2}} \rfloor = n$ for every $1 \leq n \leq 21$, except $n = 10, 12, 18$.

\begin{figure}[h]
\caption{Graph of $p$ versus $\delta(p)$ for $p \leq 300$, superimposed with the function $f(x) = \lfloor \sqrt[3]{\frac{x}{2}}\rfloor$.}
\includegraphics[width=3in]{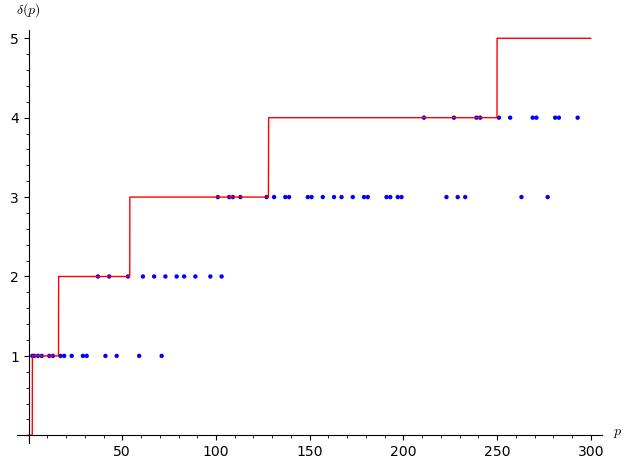}
\label{figure0}
\end{figure}

\begin{figure}[h]
\caption{Graph of $p$ versus $\delta(p)$ for $p \leq 4600$, superimposed with the function $f(x) = \lfloor \sqrt[3]{\frac{x}{2}} \rfloor$.}
\includegraphics[width=3.02in]{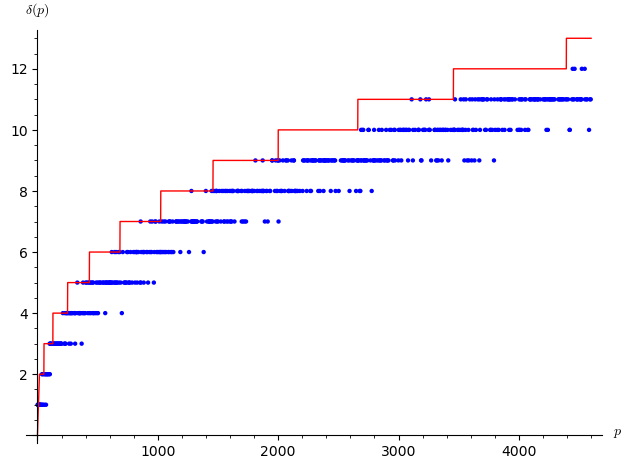}
\label{figure1}
\end{figure}

\begin{figure}
\caption{Graph of $p$ versus $\delta(p)$ for $p \leq 22,000$, superimposed with the function $f(x) = \lfloor \sqrt[3]{\frac{x}{2}}\rfloor$.}
\includegraphics[width=3.02in]{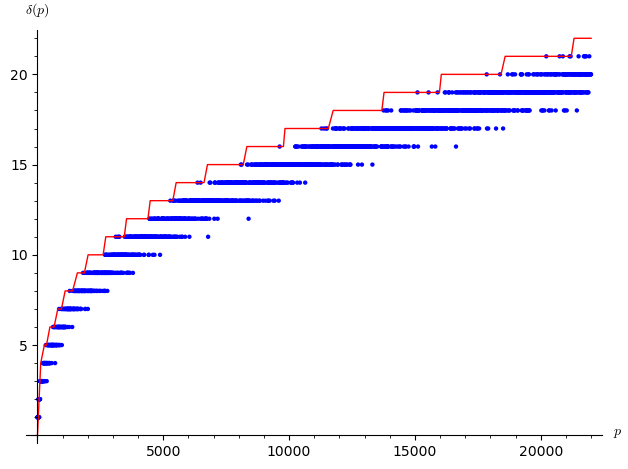}
\label{figure2}
\end{figure}

\subsection{Conjectures}\label{sec:conjectures}

In light of the data we have obtained, we end with five conjectures on the function $\delta$ that we have defined, and which we hope spur more research on the topic.

\begin{conjecture}
Let $\delta$ be the function whose domain is the set $\mathcal{E}$ of isomorphism classes of supersingular elliptic curves defined over $\overline{\mathbb{F}}_p$ for all primes $p$ assigning to an isomorphism class of curves $E$ the value $\delta_E$. Then for each $n \in \mathbb{Z}_{>0}$, there is $E$ with $\delta_E = n$:
\begin{equation*}
\Im \left( \substack{ \delta \colon \mathcal{E} \to \mathbb{Z}_{>0} \\ \quad E \mapsto \delta_E} \right) = \mathbb{Z}_{>0}.
\end{equation*}
In other words, the function $\delta$ on isomorphism classes of elliptic curves is surjective.
\end{conjecture}

In fact, something much stronger is probably the case: For example, fixing the prime $p = 234,959$ for which $\delta(p) = 48 = \lfloor \sqrt[3]{\frac{p}{2}} \rfloor$, we find that for each $1 \leq n \leq 48$ there is $E$ defined over $\overline{\mathbb{F}}_{234959}$ with $\delta_E = n$. However we find that the generalization of such a strong statement (that is, that for each prime $p$ and each $1 \leq n \leq \delta(p)$ there is $E$ defined over $\overline{\mathbb{F}}_p$ with $\delta_E = n$) is false, as witnessed by $p = 101$: We have $\delta(101) = 3$ but no supersingular elliptic curve defined over $\overline{\mathbb{F}}_{101}$ with $\delta_E = 2$. However, we propose the following weaker statement:

\begin{conjecture}\label{conj:weakgap}
For each $n \in \mathbb{Z}_{>0}$, there exists a prime $p$ such that for all $1 \leq m \leq n$, there is a supersingular elliptic curve $E$ defined over $\overline{\mathbb{F}}_p$ with $\delta_E = m$.
\end{conjecture}

This conjecture is verified for each $1 \leq n \leq 48$ by the prime $p = 234,959$, but for small values of $n$ we can find the least prime $p$ for which Conjecture \ref{conj:weakgap} is true, and we list those in Table \ref{table:weakgapconj} below.
\begin{table}[h]
\caption{Least prime $p$ for which Conjecture \ref{conj:weakgap} is verified for $1 \leq n \leq 21$}
\begin{tabular}{ c | l }
$n$ & Least $p$ \\
\hline
1 & 2 \\
2 & 37\\
3 & 107 \\ 
4 & 211\\
5 & 331 \\
6 & 617 \\
7 & 857\\
8 & 1279\\
9 & 1811\\
10 & 2699 \\ 
11 & 3181  
\end{tabular}
\qquad
\begin{tabular}{ c | l }
$n$ & Least $p$ \\
\hline
12 & 4447\\
13 & 5387 \\ 
14 & 6361\\
15 & 8081\\
16 & 10289 \\ 
17 & 11287\\
18 & 13757 \\
19 & 15101 \\
20 & 18691\\ 
21 & 20879 \\ 
\dots & \dots 
\end{tabular}
\label{table:weakgapconj}
\end{table}

We now turn our attention to the values of $\delta$ as a function on primes, rather than a function on isomorphism classes of supersingular elliptic curves.

\begin{conjecture}
\label{conj:deltapsurj}
Let $\delta$ be the function whose domain is the set $\mathbb{P}$ of prime numbers assigning to a prime $p$ the value $\delta(p)$. Then for each $n \in \mathbb{Z}_{>0}$, there is $p\in \mathbb{P}$ with $\delta(p) = n$:
\begin{equation*}
\Im \left( \substack{ \delta \colon \mathbb{P} \to \mathbb{Z}_{>0} \\ \quad p \mapsto \delta(p)} \right) = \mathbb{Z}_{>0}.
\end{equation*}
In other words, the function $\delta$ on primes is surjective.
\end{conjecture}

This conjecture is verified numerically for $1 \leq n \leq 21$, as shown in Figure \ref{figure2}. In addition, Table \ref{table:mediumwisdeprimes} shows that the conjecture is verified for $22 \leq n \leq 50$ except $n = 28, 31, 40, 43, 44.$ For these remaining values, we have that $\delta(51,479) = 28, \, \delta(68,113) = 31,\,  \delta(140,827) = 40,\, \delta(174,389) = 43$ and $\delta(188,011) = 44$, and therefore Conjecture \ref{conj:deltapsurj} is true for all $n \leq 50$.

Now we turn our attention to the set of wisde primes:
\begin{equation*}
\mathcal{W} = \left\{ p \in \mathbb{P} : \delta(p) = \left\lfloor \sqrt[3]{\frac{p}{2}} \right\rfloor \right\}.
\end{equation*}

\begin{conjecture}\label{conj:Winfinite}
The set of wisde primes $\mathcal{W}$ is infinite.
\end{conjecture}

\begin{remark}
Consider the set
\begin{equation*}
\Delta = \Im \left( \substack{ \delta \colon \mathcal{W} \to \mathbb{Z}_{>0} \\ \quad p \mapsto \delta(p)} \right).
\end{equation*}
Conjecture \ref{conj:Winfinite} is equivalent to the conjecture that $\Delta$ is infinite. From the data of Table \ref{table:smallwisdeprimes}, we have that $10, 12, 18 \not \in \Delta$,  
and from the data of Table \ref{table:mediumwisdeprimes}, we have $28, 31, 40, 43, 44 \not \in \Delta$. All other values $1 \leq n \leq 50$ belong to $\Delta$.
\end{remark}

\begin{figure}[h]
\caption{Graph of $p$ versus $\delta(p)$ for $93,312 \leq p \leq 101,306$, superimposed with the function $f(x) = \lfloor \sqrt[3]{\frac{x}{2}}\rfloor$.}
\includegraphics[width=3.02in]{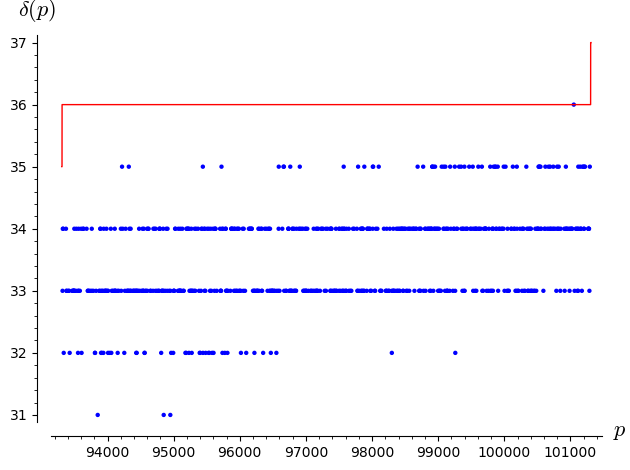}
\label{figure3}
\end{figure}

Finally, the data of Figures \ref{figure2} and \ref{figure3} suggest that $\delta(p)$ remains quite close to the value of $\left \lfloor \sqrt[3]{\frac{p}{2}} \right\rfloor$ as $p$ grows. Indeed for primes $p \leq 22,000$ we have $\left \lfloor \sqrt[3]{\frac{p}{2}} \right\rfloor - \delta(p) \leq 4$, and for primes $93,312 \leq p \leq 101,306$, we have $\left \lfloor \sqrt[3]{\frac{p}{2}} \right\rfloor - \delta(p) \leq 5$. Therefore it seems reasonable to guess that $\left \lfloor \sqrt[3]{\frac{p}{2}} \right\rfloor - \delta(p)$ grows like $\log(p)$:
\begin{conjecture}
We have that
\begin{equation*}
\limsup_{p \to \infty} \frac{\left \lfloor \sqrt[3]{\frac{p}{2}} \right\rfloor - \delta(p)}{\log p} < \infty.
\end{equation*}
\end{conjecture}

\bibliographystyle{amsalpha}
\bibliography{biblio}
\end{document}